\documentclass{amsart}
\usepackage{graphicx}
\usepackage[dvipdfm, colorlinks, linkcolor=blue, anchorcolor=green, citecolor=red]{hyperref}
\usepackage{fourier}
\newtheorem{theorem}{Theorem}[section]

\newtheorem{lemma}[theorem]{Lemma}
\newtheorem{proposition}[theorem]{Proposition}
\theoremstyle{definition}

\theoremstyle{remark}

\numberwithin{equation}{section}

\begin{document}
\title[A transcendental function invariant of virtual knots]{A transcendental function invariant of virtual knots}%
\author{Zhiyun Cheng}%
\address{School of Mathematical Sciences, Beijing Normal University, Laboratory of Mathematics and Complex Systems, Ministry of
Education, Beijing 100875, China}
\address{Department of Mathematics, The George Washington University, Washington, DC 20052, U.S.A.}%
\email{czy@bnu.edu.cn}%

\thanks{The author is supported by NSFC 11301028}%
\begin{abstract}
In this work we describe a new  invariant of virtual knots. We show that this transcendental function invariant generalizes several polynomial invariants of virtual knots, such as the writhe polynomial \cite{Che2013}, the affine index polynomial \cite{Kau2013} and the zero polynomial \cite{MJJ2015}.
\end{abstract}
\maketitle
\section{Introduction}
In recent years index type invariants of virtual knots have attracted a great deal of attention from researchers in virtual knot theory. Roughly speaking, this kind of virtual knot invariants are usually defined by counting the indices (also called the weights) of real crossing points in a virtual knot diagram. The first invariant of this type was introduced by Henrich in \cite{Hen2010}. It is well known that there are no degree one Vassiliev invariants in classical knot theory.  In \cite{Hen2010} Henrich constructed a sequence of Vassiliev invariants of degree one for virtual knots. By using the idea of parity discussed in \cite{Man2010}, the author defined a polynomial invariant of virtual knots, say the odd writhe polynomial \cite{Che2014}. We choose this terminology because it generalizes the odd writhe, a numerical invariant of virtual knots which was first proposed by L. Kauffman in \cite{Kau2004}. Later, these polynomial invariants were generalized independently by Y. H. Im, S. Kim and D. S. Lee \cite{Im2013}, Lena C. Folwaczny and L. Kauffman \cite{Fol2013,Kau2013}, Cheng and Gao \cite{Che2013}, S. Satoh and K. Taniguchi \cite{Shi2014}. The key point of these invariants is one can assign an index to each real crossing point such that the signed sum of crossings with the same index is preserved under the generalized Reidemeister moves. Several variations and applications of these invariants can be found in \cite{Chr2014,Jeo2015,Kim2014}.

An interesting feature of index type invariants is that one can use them to distinguish some virtual knot from its mirror image or inverse. However there also exists one obvious drawback: if a real crossing has zero index then it has no contribution to the invariant. Recently this shortcoming was improved by Myeong-Ju Jeong in \cite{MJJ2015}. In \cite{MJJ2015}, Jeong introduced the zero polynomial which focused on the real crossings with zero index. Some examples of virtual knots that have trivial writhe polynomial but nontrivial zero polynomial were given.

Inspired by Jeong's work, in this paper we will describe a new virtual knot invariant which generalizes several polynomial invariants mentioned above. For each oriented virtual knot diagram $K$ we will associate it with a transcendental function $F_K(t, s)$, which has the form $F_K(t, s)=\sum(\pm t^{g(s)})$. Here $g(s)$ is a polynomial in $s$. In Section 4 it will be found that $F_K(t, s)$ is invariant under generalized Reidemeister moves, hence it is a virtual knot invariant. We will discuss the relations between $F_K(t, s)$ and other index type polynomial invariants mentioned above. Finally, some interesting properties of $F_K(t, s)$ will be systematically studied in Section 5.

\section{Virtual knot theory and several polynomial invariants}
In this section we take a brief review of virtual knot theory and the definitions of several polynomial invariants of virtual knots.

Virtual knot theory was first introduced by L. Kauffman in \cite{Kau1999}. Roughly speaking, classical knot theory studies the embeddings of circles in 3-dimensional Euclidean space $R^3$ up to ambient isotopy. Evidently the ambient space $R^3$ can be replaced by $S^2\times I$. As an generalization of the classical knot theory, virtual knot theory studies the embeddings of circles in $\Sigma_g\times I$ up to ambient isotopy, homeomorphisms of $\Sigma_g$ and the addition or substraction of empty handles, here $\Sigma_g$ denotes the closed orientable surface with genus $g$. For simplicity we say two embeddings are \emph{stably equivalent} if one can be obtained from the other one by some operations above. When $g=0$ virtual knot theory recovers the classical knot theory \cite{Kau1999,Kup2003}.

Another useful way to understand virtual knots is to regard them as the realizations of Gauss diagrams. Given a classical knot diagram, one can find a unique Gauss diagram of it. However for a given Gauss diagram sometimes one can not realize it as a classical knot diagram. In order to settle this problem, one needs to add some virtual crossings on the knot diagram, although these virtual crossings are not indicated on the given Gauss diagram. Since the way of adding virtual crossings is not unique, besides of the original Reidemeister moves we need to add some virtual Reidemeister moves to remove the arbitrariness \cite{Gou2000}. See Figure \ref{figure1}.
\begin{figure}
\centering
\includegraphics{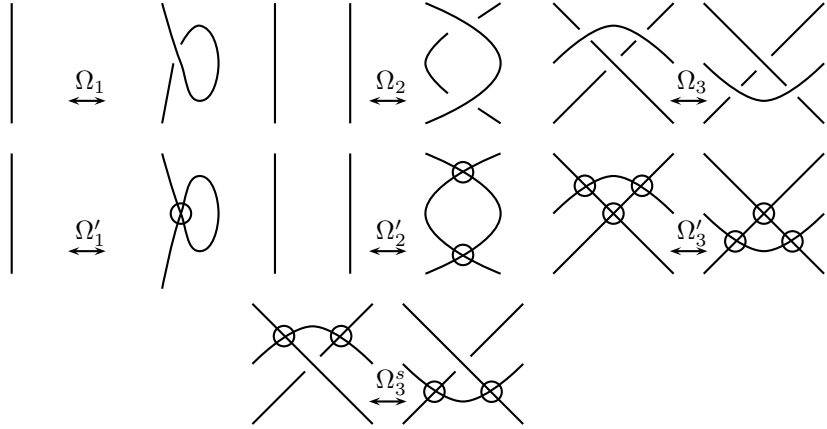}\\
\caption{Generalized Reidemeister moves}\label{figure1}
\end{figure}

Formally speaking, a virtual knot diagram can be obtained from a classical knot diagram by replacing some real crossings with virtual crossings. Usually a virtual crossing is represented by a small circle placed around the crossing point. Two virtual knot diagrams are \emph{equivalent} if they can be connected by a finite sequence of generalized Reidemeister moves. \emph{Virtual knots} can be defined as the equivalence classes of virtual knot diagrams. For a given virtual knot diagram $K$, consider the one point compactification of the plane where the diagram locates in, then one gets a 2-sphere of the diagram. Let us attach a 1-handle to a virtual crossing and regard this virtual crossing as an overpass locally. After interpreting all virtual crossings like this and thickening the surface, we will obtain an embedding of $S^1$ in $\Sigma_{c_v(K)}\times I$ where $c_v(K)$ is the number of virtual crossing points in $K$. An important fact is these two interpretations of virtual knots coincides.
\begin{theorem}{\cite{Kau1999}}
Two virtual knot diagrams are equivalent if and only if their corresponding surface embeddings are stably equivalent.
\end{theorem}

Since one motivation of introducing virtual knots is to realize arbitrary Gauss diagram, sometimes it is more convenient to study the corresponding Gauss diagram than to study the virtual knot diagram. Let us give a short review of the definition of the Gauss diagram. Let $K$ be a virtual knot diagram, the \emph{Gauss diagram} of $K$, written $G(K)$, is an oriented (counterclockwise) circle where the preimages of each real crossing are indicated. For the two preimages of a real crossing we add a chord connecting them, which is directed from the overcrossing to the undercrossing. Finally each chord is associated with a sign according to the writhe of the corresponding real crossing. One simple example is given in Figure \ref{figure2}. We remark that virtual Reidemeister moves $\{\Omega_1', \Omega_2', \Omega_3', \Omega_3^s\}$ have no effect on the Gauss diagram. Therefore as an advantage, when we study virtual knots from the viewpoint of Gauss diagrams we only need to consider the classical Reidemeister moves $\{\Omega_1, \Omega_2, \Omega_3\}$.
\begin{figure}
\centering
\includegraphics{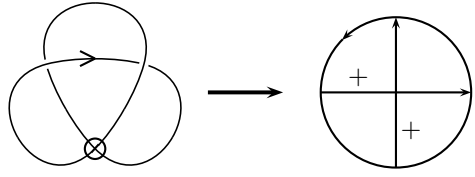}\\
\caption{An example of Gauss diagram}\label{figure2}
\end{figure}

In order to classify virtual knots, one needs to introduce some virtual knot invariants. Some classical knot invariants can be extended directly to virtual knots, such as the Jones polynomial and the knot quandle. Later these invariants were generalized using the ``virtual structure" of virtual knots. For Jones polynomial, based on Kauffman's approach to it, after smoothing all real crossings one can assign a weight to each circle. The state-sum gives a generalization of Jones polynomial for virtual knots. This were independently finished by Y. Miyazawa \cite{Miy2006,Miy2008} and H. Dye, L. Kauffman \cite{Dye2009}, now known as the Miyazawa polynomial and the arrow polynomial. On the other hand, the idea of quandle was generalized to some complicated algebraic structure, such as the biquandle \cite{Fen2004} and virtual biquandle \cite{Kau2005}. We refer the reader to \cite{Kau2012,Man2012} for more details on these developments.

Now we recall the definition of the writhe polynomial introduced in \cite{Che2013}. Later we will show how to obtain other index type virtual knot invariants from this polynomial invariant.

Let $K$ be a virtual knot diagram and $G(K)$ the corresponding Gauss diagram. Since there is a bijection between the real crossings of $K$ and the chords of $G(K)$, we will use the same letter to refer a real crossing point and the corresponding chord. Choose a chord $c$ in $G(K)$, now let us assign an index to it, which will play an important role in the definition of the writhe polynomial. Let $r_+$ $(r_-)$ denotes the number of positive (negative) chords crossing $c$ from left to right, let $l_+$ $(l_-)$ denotes the number of positive (negative) chords crossing $c$ from right to left (See Figure \ref{figure3}). Following \cite{Che2013}, we define the \emph{index} of $c$ as
\begin{center}
Ind$(c)=r_+-r_--l_++l_-$.
\end{center}
We mention some similar indices appeared in the literature. In \cite{Hen2010} Henrich defined an \emph{intersection index} for each chord, which equals to the absolute value of the index here. In \cite{Dye2013} Dye introduced a \emph{parity mapping} from the chords of a Gauss diagram to $\mathbb{Z}$, which is exactly the inverse of the index used here.
\begin{figure}
\centering
\includegraphics{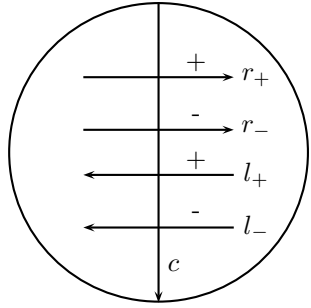}\\
\caption{Index of a chord}\label{figure3}
\end{figure}

Here we list some useful properties of Ind$(c)$, which can be easily verified.
\begin{enumerate}
  \item If $c$ is isolated, i.e. no other chord has intersection with $c$, then Ind$(c)=0$.
  \item The two crossings appeared in $\Omega_2$ have the same index.
  \item The indices of the three crossings appeared in $\Omega_3$ are invariant under $\Omega_3$.
  \item $\Omega_i$ $(i=1, 2, 3)$ preserves the indices of chords that do not appear in $\Omega_i$ $(i=1, 2, 3)$.
  \item If $K$ contains no virtual crossings, then every crossing of $K$ has index zero.
  \item Ind$(c)$ is invariant under switching some other real crossings.
\end{enumerate}

Due to the properties above, one can define a numerical invariant of virtual knots as below
\begin{center}
$Q_K=\sum\limits_{\text{Ind}(c_i)\neq0}w(c_i)$  $(c_i\in C_r(K))$,
\end{center}
where $w(c_i)$ denotes the writhe of $c_i$ and $C_r(K)$ denotes the set of real crossings of $K$.
Note that
\begin{center}
$-Q_K=\sum\limits_{\text{Ind}(c_i)=0}w(c_i)-w(K)$ $(c_i\in C_r(K))$,
\end{center}
where $w(K)$ denotes the writhe of $K$. The key observation is that $Q_K$ can be generalized by counting chords with a fixed index. For the sake of simplicity we write it in the form of a polynomial. We define the \emph{writhe polynomial} $W_K(t)$ to be
\begin{center}
$W_K(t)=\sum\limits_{\text{Ind}(c_i)\neq0}w(c_i)t^{\text{Ind}(c_i)}$.
\end{center}
Note that $W_K(1)=Q_K$. On the other hand, we remark that this definition is slightly different from that given in \cite{Che2013}, where the polynomial is equal to $W_K(t)\cdot t$. One can easily deduce that $W_K(t)$ is a virtual knot invariant from the properties listed above.

Before discovering the indices of chords, it was first observed by L. Kauffman \cite{Kau2004} that
\begin{center}
 $J(K)=\sum\limits_{\text{Ind}(c_i) \text{ is odd}}w(c_i)$,
\end{center}
which is named the \emph{odd writhe}, is a virtual knot invariant. Later this numerical invariant was generalized to the \emph{odd writhe polynomial}, which equals to
\begin{center}
$\sum\limits_{\text{Ind}(c_i) \text{ is odd}}w(c_i)t^{\text{Ind}(c_i)}$.
\end{center}
Again the definition we give here is also a bit different from the one we gave in \cite{Che2014}. Besides of the writhe polynomial, the odd writhe polynomial was independently generalized by several groups. We list the connections between them and the writhe polynomial below.
\begin{enumerate}
\item The \emph{parity writhe polynomial} $F_K(x, y)$  introduced  in \cite{Im2013} can be described as
\begin{center}
$\sum\limits_{\text{Ind}(c_i) \text{ is odd}}w(c_i)x^{\text{Ind}(c_i)+1}+\sum\limits_{\text{Ind}(c_i) \text{ is even}}w(c_i)y^{\text{Ind}(c_i)+1}-w(K)x$.
\end{center}
Note that replacing the variable $y$ with $x$ will not weaken the invariant. In this case, it coincides with $(W_K(x)-Q_K)x=(W_K(x)-W_K(1))x$.
\item The \emph{affine index polynomial} $P_K(t)$ defined in \cite{Kau2013} satisfies
\begin{center}
$P_K(t)=W_K(t)-Q_K=W_K(t)-W_K(1)$.
\end{center}
It is worth mentioning that in \cite{Kau2013} the algebraic structure behind the affine index polynomial was discussed. It was proved that this kind of index polynomial is essential the unique one derived from an affine linear flat biquandle with coefficients in a commutative ring without zero divisors.
\item The \emph{$n$th parity writhe} $J_n(K)$ \cite{Shi2014} is equal to the coefficient of $t^n$ in $W_K(t)$.
\end{enumerate}

Before ending this section we will give an interesting application of using writhe polynomial to detect the non-classicality of virtual knots. A long-standing open problem in knot theory is whether there exists a nontrivial knot with trivial Jones polynomial. However in virtual knot theory, it is well-known that there exist infinitely many nontrivial virtual knots with trivial Jones polynomial. More precisely, we have the following result.
\begin{theorem}{\cite{Kau2012}}
Let $K$ be a nontrivial classical knot, then there is a corresponding nontrivial virtual knot $v(K)$ with trivial Jones polynomial.
\end{theorem}
\begin{proof}
We give a sketch of the proof, see \cite{Kau2012} for more details. Choose a diagram of $K$, let us still use $K$ to denote it. We can find a set of crossings such that after switching all these crossings one will obtain a trivial knot. Now for each crossing we take a local replacement (virtualization), see Figure \ref{figure4}. Denote the new virtual knot diagram by $v(K)$. Note that switching a crossing and virtualizing a crossing have the same effect on the Jones polynomial, it follows that the Jones polynomial of $v(K)$ is trivial. However $v(K)$ has the same involutory quandle (also called \emph{kei}) as $K$, which is nontrivial since $K$ is nontrivial. Thus, $v(K)$ is a nontrivial virtual knot with unit Jones polynomial.
\end{proof}
\begin{figure}
\centering
\includegraphics{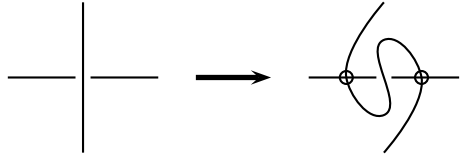}\\
\caption{Virtualization}\label{figure4}
\end{figure}

Now we can ask the following two questions.
\begin{enumerate}
\item Whether every nontrivial virtual knot that has trivial Jones polynomial can be obtained by virtualization?
\item Is $v(K)$ non-classical?
 \end{enumerate}
If both the answers are YES, then we can conclude that Jones polynomial detects the unknot. For this reason, it was suggested by L. Kauffman \cite{Kau2013} to characterize which virtualization can be detected by index type invariants. Here we give a partial answer to this question. Note that this result is in fact a special case of the main theorem in \cite{Sil2004}, which is obtained from an analysis of the knot group.
\begin{proposition}
Let $K$ be a positive (or negative) classical knot diagram with unknotting number 1. Then the corresponding virtual knot $v(K)$ is non-classical.
\end{proposition}
\begin{proof}
Without loss of generality, we assume that $K$ is positive. Choose a crossing $c$ such that switching $c$ will unknot $K$. Let us consider the Gauss diagram $G(K)$. First notice that $c$ can not be an isolated chord in $G(K)$, otherwise $c$ is a nugatory crossing and switching $c$ preserves $K$. After taking virtualization on $c$, we obtain the Gauss diagram $G(v(K))$, which can be derived from $G(K)$ by changing the sign of $c$. Since $K$ is a classical knot diagram, it follows that Ind$(c)=0$ in $G(K)$. Let $a$ $(\neq 0)$ denotes the number of positive chords crossing $c$ from left to right, then there are also $a$ positive  chords crossing $c$ from right to left. Now it is an easy exercise to check that $W_{v(K)}(t)=at^2+at^{-2}$, which is nontrivial.
\end{proof}

\section{The definition of $F_K(t, s)$}
In this section we give the definition of $F_K(t, s)$. The main idea of the construction is to replace the integer Ind$(c)$ in $W_K(t)$ by a polynomial $g_c(s)$.

Let $K$ be a virtual knot diagram and $c$ a real crossing of it. Denote the chords crossing $c$ from left to right and the chords crossing $c$ from right to left by $\{r_1, \cdots, r_n\}$ and $\{l_1, \cdots, l_m\}$ respectively. Let $\phi$ denotes the canonical quotient map from $\mathbb{Z}$ to $\mathbb{Z}_{|\text{Ind}(c)|}$. Now we define the \emph{index function} $g_c(s)$ as below
\begin{center}
$g_c(s)=\sum\limits_{i=1}^nw(r_i)s^{\phi(\text{Ind}(r_i))}-\sum\limits_{i=1}^mw(l_i)s^{\phi(-\text{Ind}(l_i))}$.
\end{center}
More precisely, if Ind$(c)=0$, in this case $\phi=id$, therefore
\begin{center}
$g_c(s)=\sum\limits_{i=1}^nw(r_i)s^{\text{Ind}(r_i)}-\sum\limits_{i=1}^mw(l_i)s^{-\text{Ind}(l_i)}$.
\end{center}
If Ind$(c)=\pm1$, then $\phi$ sends every integer into zero, hence
\begin{center}
$g_c(s)=\sum\limits_{i=1}^nw(r_i)-\sum\limits_{i=1}^mw(l_i)=\text{Ind}(c)$.
\end{center}
In general, $g_c(s)$ takes values in $\mathbb{Z}[s, s^{-1}]/(s^{|\text{Ind}(c)|}-1)$.

For the writhe polynomial $W_K(t)$ and the numerical invariant $Q_K$, we define a refined version for each of them by letting
\begin{center}
$W_K(t, s)=\sum\limits_{\text{Ind}(c_i)\neq0}w(c_i)t^{g_{c_i}(s)}$, and $Q_K(t, s)=w(K)-\sum\limits_{\text{Ind}(c_i)=0}w(c_i)t^{g_{c_i}(s)}$.
\end{center}
Now we define the function $F_K(t, s)$ as
\begin{center}
$F_K(t, s)=\sum\limits_{c_i}w(c_i)t^{g_{c_i}(s)}-w(K)=W_K(t, s)-Q_K(t, s).$
\end{center}
We remark that according to the definition of the index function, it is easy to see that $g_{c_i}(1)=\text{Ind}(c_i)$. Hence $g_{c_i}(s)$ in fact takes values in $\mathbb{Z}[s, s^{-1}]/(s^{|g_{c_i}(1)|}-1)$, which guarantees the definition above is well-defined.
\begin{theorem}
$F_K(t, s)$ is a virtual knot invariant.
\end{theorem}

Here we make some remarks on this invariant. First, in fact it is not essential to use two variables $t$ and $s$. One can easily find that replacing $s$ with $t$ will not weaken this invariant. However it is more convenient to regard it as a generalization of some other polynomial invariants if we use $F_K(t, s)$ rather than $F_K(t, t)$. For example, we have
\begin{center}
 $W_K(t, 1)=W_K(t), Q_K(t, 1)=Q_K$, thus $F_K(t, 1)=P_K(t)$.
\end{center}
On the other hand, Myeong-Ju Jeong introduced the zero polynomial in \cite{MJJ2015}, which was defined by
\begin{center}
$Z_K(t)=\sum\limits_{\text{Ind}(c_i)=0}w(c_i)(t^{g_{c_i}}-1)$, where $g_{c_i}=\sum\limits_{\text{Ind}(r_j)=0}w(r_j)-\sum\limits_{\text{Ind}(l_j)=0}w(l_j)$.
\end{center}
The zero polynomial can be used to detect the non-classicality of some virtual knots that have trivial writhe polynomials. According to the definition above one can easily deduce that
\begin{center}
$Z_K(t)=Q_K-Q_K(t, 0)$.
\end{center}
Together with the discussion in Section 2, we see that $F_K(t, s)$ indeed generalises all the polynomial invariants mentioned in Section 1.

Another important thing we want to point out is, $W_K(t, s)$ and $Q_K(t, s)$ are mutually independent. Let us make a bit more explanation about this. At the beginning of this story, we have two numerical invariants of virtual knots,
\begin{center}
$Q_K=\sum\limits_{\text{Ind}(c_i)\neq0}w(c_i)$ and $\sum\limits_{\text{Ind}(c_i)=0}w(c_i)-w(K)$.
\end{center}
Obviously they are mutually the inverse of each other. After the first generalization, we have two polynomial invariants,
\begin{center}
$W_K(t)=\sum\limits_{\text{Ind}(c_i)\neq0}w(c_i)t^{\text{Ind}(c_i)}$ and $\sum\limits_{\text{Ind}(c_i)=0}w(c_i)t^{\text{Ind}(c_i)}-w(K)$.
\end{center}
We can regard the first part as the contributions from the chords with nontrivial indices, and the second part contains the contributions from the chords with index zero. However, the second part is preserved hence it can be derived from the first part. After the second generalization, now we have two transcendental function invariants of virtual knots,
\begin{center}
$W_K(t, s)=\sum\limits_{\text{Ind}(c_i)\neq0}w(c_i)t^{g_{c_i}(s)}$ and $-Q_K(t, s)=\sum\limits_{\text{Ind}(c_i)=0}w(c_i)t^{g_{c_i}(s)}-w(K)$.
\end{center}
Later we will give some examples to show that $W_K(t, s)$ and $Q_K(t, s)$ are mutually independent, i.e. neither can be derived directly from the other one. Since
\begin{center}
$W_K(1, s)=Q_K=Q_K(t, 1)$,
\end{center}
$W_K(t, s)$ and $Q_K(t, s)$ can be viewed as two different generalizations of our original numerical invariant $Q_K$.

Now we give one example to explain how to calculate $F_K(t, s)$. Let us consider the virtual knot $K$, the Gauss diagram $G(K)$ is depicted in Figure \ref{figure5}. First we write down the index of each chord
\begin{center}
$\text{Ind}(c_1)=\text{Ind}(c_2)=2, \text{Ind}(c_3)=\text{Ind}(c_4)=-1, \text{Ind}(c_5)=0$.
\end{center}
Next we list the index function of each chord
\begin{center}
$g_{c_1}(s)=s+1, g_{c_2}(s)=2, g_{c_3}(s)=g_{c_4}(s)=-1, g_{c_5}(s)=s^{-2}-s^{-1}$.
\end{center}
It follows that
\begin{center}
$F_K(t, s)=t^{s+1}-t^2+t^{-1}-t^{-1}+t^{s^{-2}-s^{-1}}-1=t^{s+1}-t^2+t^{s^{-2}-s^{-1}}-1$.
\end{center}
 We remark that in this example the writhe polynomial and zero polynomial are both trivial.
\begin{figure}
\centering
\includegraphics{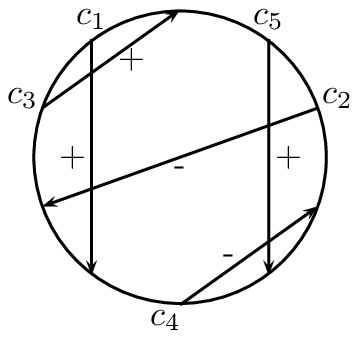}\\
\caption{$G(K)$}\label{figure5}
\end{figure}
\begin{figure}
\centering
\includegraphics{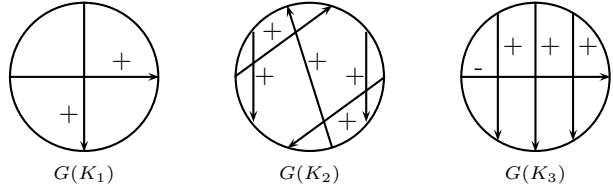}\\
\caption{Gauss diagrams of $K_1, K_2, K_3$}\label{figure6}
\end{figure}

The second example concerns the three virtual knots described in Figure \ref{figure6}. Direct calculation shows that
\begin{flalign*}
&W_{K_1}(t, s)=t+t^{-1}, Q_{K_1}(t, s)=2,\\
&W_{K_2}(t, s)=t+t^{-1}, Q_{K_2}(t, s)=-t^{1-s^{-1}}-t^{s^{-1}-1}+4,\\
&W_{K_3}(t, s)=3t^{-1}-t^{-3s}, Q_{K_3}(t, s)=2.
\end{flalign*}
Note that
\begin{center}
$W_{K_1}(t, s)=W_{K_2}(t, s), Q_{K_1}(t, s)=Q_{K_3}(t, s)$,
\end{center}
but
\begin{center}
$Q_{K_1}(t, s)\neq Q_{K_2}(t, s), W_{K_1}(t, s)\neq W_{K_3}(t, s)$.
\end{center}
This implies that in general $W_K(t, s)$ and $Q_K(t, s)$ are mutually independent, as we claimed before.

Theorem 2.3 tells us that writhe polynomial can be used to detect the non-classicality of the virtualization of some classical knots. However there exists some classical knot whose virtualization has the trivial writhe polynomial. The third example was suggested by L. Kauffman in \cite{Kau2013}. The classical knot $K$ illustrated in Figure \ref{figure12} has unknotting number 3. For example, switching crossings $\{c_1, c_2, c_3\}$ yields the trivial knot. Straightforward calculation shows that $W_{v(K)}(t)=0$ but $F_{v(K)}(t, s)=t^{-2s^2-2}-t^{-3s^2-1}-t^{s^2-1}+t^{s^{-4}-1}$, which means $v(K)$ is non-classical and hence nontrivial.
\begin{figure}
\centering
\includegraphics{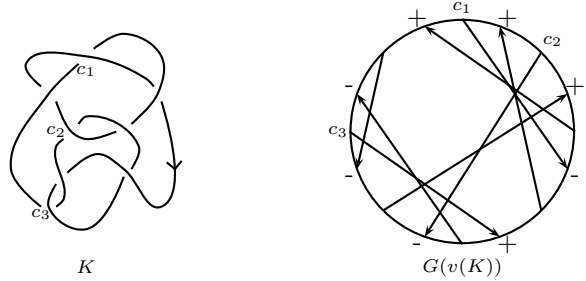}\\
\caption{A classical knot and the Gauss diagram of its virtulization}\label{figure12}
\end{figure}

\section{The proof of Theorem 3.1}
In this section we give the proof of Theorem 3.1. Before doing this, we need a simple lemma.
\begin{lemma}
Let $a, b, c$ be the three crossings of the virtual knot diagram locally described on the left side of Figure \ref{figure7}. Then Ind$(a)$-Ind$(b)$+Ind$(c)=0$.
\end{lemma}
\begin{proof}
According to the connecting ways outside of the local diagram, there are two possibilities of the Gauss diagram of $K$, say $G(K)_1$ and $G(K)_2$ (see Figure \ref{figure7}).

Let us consider $G(K)_1$. Due to the sixth property of index we listed in Section 2, without loss of generality, we assume there are $x, y, z$ positive chords intersects $a$ and $c$, $b$ and $c$, $a$ and $b$ respectively. See the middle graph of Figure \ref{figure7} for the directions of these chords. It follows directly that
\begin{center}
Ind$(a)=x-z$, Ind$(b)=-y-z$, Ind$(c)=-x-y$,
\end{center}
which implies that Ind$(a)$-Ind$(b)$+Ind$(c)=0$. The proof for $G(K)_2$ is analogous.
\end{proof}
\begin{figure}
\centering
\includegraphics{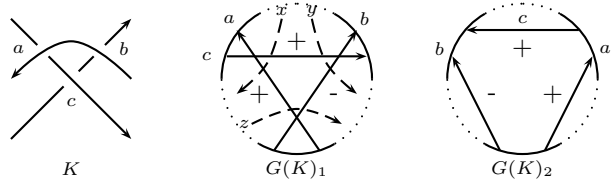}\\
\caption{Two possibilities of the Gauss diagram of $K$}\label{figure7}
\end{figure}

\begin{figure}
\centering
\includegraphics{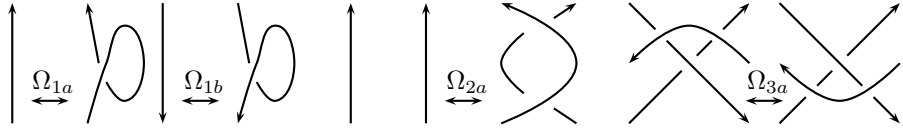}\\
\caption{A generating set of Reidemeister moves}\label{figure8}
\end{figure}
Now we give the proof of Theorem 3.1.
\begin{proof}
It suffices to check that $F_K(t, s)$ is invariant under all Reidemeister moves illustrated in Figure \ref{figure8} \cite{Pol2010}.

The only chord involved in $\Omega_{1a}$ is isolated. Therefore it has index zero and the indices of other chords are invariant. Consequently the index function of the isolated chord is zero, and the index functions of other chords are invariant under $\Omega_{1a}$. It is obvious that $F_K(t, s)$ is preserved under $\Omega_{1a}$. Similarly one can prove that $\Omega_{1b}$ also preserves $F_K(t, s)$.

For $\Omega_{2a}$, first we notice that the two chords involved have the same index but different writhes. Secondly, if a chord intersects one of them then it also has nonempty intersection with the other one. For these reasons, the two chords involved in $\Omega_{2a}$ have the same index functions but different writhes. It follows that the contributions from these two chords  to $F_K(t, s)$ will cancel out. On the other hand, these two chords will not affect the index function of any other chord, which implies that the contribution from any chord which is not involved in $\Omega_{2a}$ is invariant under $\Omega_{2a}$. We conclude that $F_K(t, s)$ is invariant under $\Omega_{2a}$.

Now let us consider the third Reidemeister move $\Omega_{3a}$. As we mentioned in Lemma 4.1, there exist two possibilities for the Gauss diagram of the whole knot diagram. We only consider the one corresponding to $G(K)_1$, the proof for $G(K)_2$ is quite similar.

The variation of the Gauss diagram $G(K)_1$ under $\Omega_{3a}$ is depicted in Figure \ref{figure9}. Recall the properties (3) and (4) mentioned in Section 2, the index of each chord is preserved under $\Omega_{3a}$. For any chord $d$ which is not involved in $\Omega_{3a}$,  it is not difficult to observe that the index function $g_d(s)$ is invariant. Therefore it suffices to consider the behaviors of $g_a(s), g_b(s), g_c(s)$ under $\Omega_{3a}$.

For the Gauss diagram on the left side of Figure \ref{figure9}, we have
\begin{flalign*}
&g_a(s)=a(s)+s^{\phi_a(-\text{Ind}(b))}-s^{\phi_a(-\text{Ind}(c))},\\
&g_b(s)=b(s)+s^{\phi_b(\text{Ind}(a))}-s^{\phi_b(-\text{Ind}(c))},\\
&g_c(s)=c(s)+s^{\phi_c(\text{Ind}(a))}-s^{\phi_c(\text{Ind}(b))}.
\end{flalign*}
Here $a(s), b(s), c(s)$ count the contributions from those chords not involved in $\Omega_{3a}$ to $g_a(s), g_b(s), g_c(s)$ respectively, $\phi_a, \phi_b, \phi_c$ denote the projection from $\mathbb{Z}$ to $\mathbb{Z}_{|\text{Ind}(a)|}, \mathbb{Z}_{|\text{Ind}(b)|}, \mathbb{Z}_{|\text{Ind}(c)|}$ respectively. For the Gauss diagram on the right side of Figure \ref{figure9}, one computes that
\begin{flalign*}
g_a(s)=a(s), g_b(s)=b(s), g_c(s)=c(s).
\end{flalign*}
Recall that Lemma 4.1 tells us Ind$(a)$-Ind$(b)$+Ind$(c)$=0. Thus
\begin{flalign*}
&s^{\phi_a(-\text{Ind}(b))}-s^{\phi_a(-\text{Ind}(c))}=s^{\phi_a(-\text{Ind}(a)-\text{Ind}(c))}-s^{\phi_a(-\text{Ind}(c))}=0,\\
&s^{\phi_b(\text{Ind}(a))}-s^{\phi_b(-\text{Ind}(c))}=s^{\phi_b(\text{Ind}(b)-\text{Ind}(c))}-s^{\phi_b(-\text{Ind}(c))}=0,\\
&s^{\phi_c(\text{Ind}(a))}-s^{\phi_c(\text{Ind}(b))}=s^{\phi_c(\text{Ind}(b)-\text{Ind}(c))}-s^{\phi_c(\text{Ind}(b))}=0,\\
\end{flalign*}
The proof is completed.
\end{proof}
\begin{figure}
\centering
\includegraphics{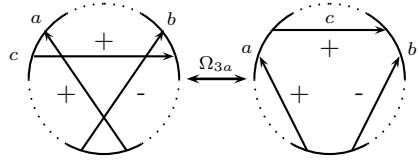}\\
\caption{The variation of $G(K)_1$ under $\Omega_{3a}$}\label{figure9}
\end{figure}

\section{Some properties of $F_K(t, s)$}
In this section we listed some basic properties of $F_K(t, s)$. First of all, similar to the other index type invariants, $F_K(t, s)$ is good at detecting the non-classicality of virtual knots.
\begin{proposition}
Let $K$ be a virtual knot diagram, if all real crossings of $K$ have trivial indices, then $F_K(t, s)=0$. In particular, classical knots have trivial $F_K(t, s)$.
\end{proposition}
\begin{proof}
Choose a real crossing $c_i$, since all other crossings have trivial indices, then $g_{c_i}(s)=\text{Ind}(c_i)=0$. It follows that
\begin{center}
$F_K(t, s)=\sum\limits_{c_i}w(c_i)t^{g_{c_i}(s)}-w(K)=\sum\limits_{c_i}w(c_i)-w(K)=0$.
\end{center}
\end{proof}

\begin{proposition}
Given a virtual knot diagram $K$, let $m(K)$ denote the virtual knot diagram with all real crossings switched, let $r(K)$ be the virtual knot diagram with the orientation reversed. Then
\begin{center}
$F_{m(K)}(t, s)=-F_K(t^{-1}, s^{-1})$ and $F_{r(K)}(t, s)=F_K(t^{-1}, s)$.
\end{center}
\end{proposition}
\begin{proof}
Choose a chord $c$ in $G(K)$, we denote the corresponding chords in $G(m(K))$ and $G(r(K))$ by $c'$ and $c''$ respectively. According to the definition of chord index, it is not difficult to observe that
\begin{center}
$w(c)=-w(c')=w(c'')$ and $\text{Ind}(c)=-\text{Ind}(c')=-\text{Ind}(c'')$.
\end{center}
It follows that
\begin{center}
$g_{c_i'}(s)=-g_{c_i}(s^{-1})$ and $g_{c_i''}(s)=-g_{c_i}(s)$.
\end{center}
Then if $F_K(t, s)=\sum\limits_{c_i}w(c_i)t^{g_{c_i}(s)}-w(K)$, we have
\begin{center}
$F_{m(K)}(t, s)=\sum\limits_{c_i'}w(c_i')t^{g_{c_i'}(s)}-w(m(K))=-\sum\limits_{c_i}w(c_i)t^{-g_{c_i}(s^{-1})}+w(K)=-F_K(t^{-1}, s^{-1})$,
\end{center}
 and
 \begin{center}
$F_{r(K)}(t, s)=\sum\limits_{c_i''}w(c_i'')t^{g_{c_i''}(s)}-w(r(K))=\sum\limits_{c_i}w(c_i)t^{-g_{c_i}(s)}-w(K)=F_K(t^{-1}, s)$.
 \end{center}
\end{proof}

Recall that the real crossing number of a virtual knot diagram is the number of real crossings of this diagram. In virtual knot theory, the \emph{real crossing number} $c_r(K)$ of a virtual knot $K$ is the smallest number of real crossings of any diagram of $K$. Obviously $c_r(K)=0$ if and only if $K$ is trivial, and there is no virtual knot with $c_r(K)=1$. According to the definition of $F_K(t, s)$, it is easy to see that $F_K(t, s)$ gives a natural lower bounder for $c_r(K)$.
\begin{proposition}
For a given virtual knot $K$, if $F_K(t, s)$ has the form $F_K(t, s)=\sum\limits_{g_i(s)}a_{g_i(s)}t^{g_i(s)}+b$, where $a_{g_i(s)}, b\in \mathbb{Z}$ and $g_i(s)\in \mathbb{Z}[s, s^{-1}]-\{0\}$, then $c_r(K)\geq \sum\limits_{g_i(s)}|a_{g_i(s)}|$.
\end{proposition}

For example, let us consider the virtual knot $K$ described in Figure \ref{figure11}. Direct calculation shows that $W_K(t)=t^2+t+t^{-3}$ and $Z_K(t)=0$. The writhe polynomial tells us that $c_r(K)\geq3$. In contrast we have $F_K(t, s)=t^{s+1}+t+t^{s^{-3}-s^{-2}}+t^{-s^2-s-1}-4$, which implies $c_r(K)\geq4$. Since the Gauss diagram in Figure \ref{figure11} has only 4 chords, we conclude that $c_r(K)=4$.
\begin{figure}
\centering
\includegraphics{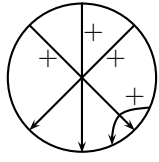}\\
\caption{A virtual knot with $c_r(K)=4$}\label{figure11}
\end{figure}

Next we study the behavior of $F_K(t, s)$ under a connected sum. Because of the forbidden move, in general the connected sum is not well-defined for virtual knots. For example, the connected sum of two trivial knots maybe nontrivial \cite{Kis2004}. Similar to other index type invariants of virtual knots, $F_K(t, s)$ is additive under a connected sum.
\begin{proposition}
Let $K_1$ and $K_2$ be two virtual knots, then $F_{K_1\#K_2}(t, s)=F_{K_1}(t, s)+F_{K_2}(t, s)$. Here $K_1\#K_2$ denotes one of the connected sums of $K_1$ and $K_2$ with an arbitrarily chosen connection place.
\end{proposition}
\begin{proof}
It suffices to notice that in $G(K_1\#K_2)$ the chords from $G(K_1)$ have no intersections with the chords from $G(K_2)$. The result follows directly from the definition.
\end{proof}

This proposition reveals a shortcoming of $F_K(t, s)$. If a nontrivial virtual knots is a connected sum of two trivial knots, then we can not detect it via calculating the $F_K(t, s)$ of it. For example, the Kishino Knot depicted in Figure \ref{figure10}  has trivial $F_K(t, s)$. However it is well known that the Kishino Knot is nontrivial (in fact even the corresponding flat virtual knot of it is nontrivial), which can be detected by, for example, the arrow polynomial \cite{Kau2012}. We remark that the Kishino Knot also provides a counterexample to the first question we listed in the end of Section 2, since the involutory quandle of Kishino Knot is trivial. Further, if each crossing of $K$ has index zero, Proposition 5.1 tells us that $F_K(t, s)=0$. The Kishino Knot is a special case of this. For writhe polynomial, if one chord has index zero then it has no contribution to the writhe polynomial. For $F_K(t, s)$, the contribution comes from index zero chord may not be zero. However $F_K(t, s)$ is powerless to distinguish a virtual knot from the unknot if each chord of it has index zero. It is a challenging question to deal with this kind of knots by introducing some extended index invariants. The three loop isotopy invariant \cite{Chr2014} defined by assigning a weight to a pair of  non-intersecting chords may offer some hints to this problem.
\begin{figure}
\centering
\includegraphics{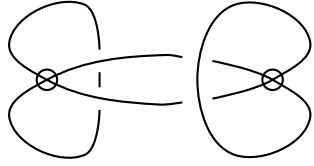}\\
\caption{The Kishino Knot}\label{figure10}
\end{figure}

Finally, let us take a moment to discuss the degree of the invariant we introduced in the present paper. In virtual knot theory, there are mainly two kinds of finite-type invariants. One was introduced by L. Kauffman in his introductory paper of virtual knot theory \cite{Kau1999}, the other one was proposed by Goussarov, Polyak, and Viro in \cite{Gou2000}. Here we will focus on the Kauffman finite-type invariants, for the GPV finite-type invariants we refer readers to \cite{Gou2000}. We remark that if a virtual knot invariant is a GPV finite-type invariant then it must be a Kauffman finite-type invariant. On the other hand, if a knot diagram is classical, in this case Kauffman finite-type invariants are exactly the classical finite-type invariants \cite{Bir1993}. Recall that a virtual knot invariant $f$ taking values in an abelian group is called a \emph{finite-type invariant of degree $\leq n$}, if for any virtual knot diagram $K$ with $n+1$ singular crossings we have
\begin{center}
$\sum\limits_{\sigma\in\{0, 1\}^{n+1}}(-1)^{|\sigma|}f(K_{\sigma})=0$.
\end{center}
Here $\sigma$ runs over all $(n+1)$-tuples of zeros and ones, $|\sigma|$ denotes the number of ones in $\sigma$ and $K_{\sigma}$ is obtained from $K$ by replacing the $i$-th singular crossing with a positive (negative) crossing if the $i$-th position of $\sigma$ is zero (one). The minimal $n$ is said to be the \emph{degree} of $f$.
\begin{proposition}
The invariant $F_K(t, s)$ is a finite-type invariant of degree one.
\end{proposition}
\begin{proof}
It is sufficient to show that for any virtual knot diagram $K$ with two singular crossings, we have $F_{K_{00}}(t, s)-F_{K_{01}}(t, s)-F_{K_{10}}(t, s)+F_{K_{11}}(t, s)=0$.

Let $a_{\sigma}, b_{\sigma}$ denote the two real crossings obtained from the singular crossings in $K_{\sigma}$ $(\sigma\in\{00, 01, 10, 11\})$. First we assume that the chord corresponding to $a_{\sigma}$ has no intersection with the chord corresponding to $b_{\sigma}$. It is not difficult to observe that $\text{Ind}(a_{00})=\text{Ind}(a_{01})=-\text{Ind}(a_{10})=-\text{Ind}(a_{11})$, $\text{Ind}(b_{00})=\text{Ind}(b_{10})=-\text{Ind}(b_{01})=-\text{Ind}(b_{11})$, and the index of any other chord in $K_{\sigma}$ $(\sigma\in\{00, 01, 10, 11\})$ is equivalent. It follows that for any real crossing $c$ in $K$, the corresponding chords in $K_{\sigma}$ $(\sigma\in\{00, 01, 10, 11\})$ have the same index function. Therefore we only need to consider the contributions from $a_{\sigma}, b_{\sigma}$ to $F_{K_{00}}(t, s)-F_{K_{01}}(t, s)-F_{K_{10}}(t, s)+F_{K_{11}}(t, s)$, which equals
\begin{center}
$t^{g_a(s)}+t^{g_b(s)}-t^{g_a(s)}+t^{-g_b(s^{-1})}+t^{-g_a(s^{-1})}-t^{g_b(s)}-t^{-g_a(s^{-1})}-t^{-g_b(s^{-1})}=0$.
\end{center}
Here $g_a(s)$ $(g_b(s))$ denotes the index function of $a_{00}$ $(b_{00})$ in $K_{00}$.

When the corresponding chord of $a_{\sigma}$ intersects the chord corresponding to $b_{\sigma}$ in $K_{\sigma}$, all arguments above are still valid. The details are left to the reader.
\end{proof}

Roughly speaking, the reason why $F_K(t, s)$ is a finite-type invariant of degree one is that $F_K(t, s)$ is a weighted sum of each chord in the Gauss diagram. If one wants to define a finite-type invariant of higher degree, a sensible idea is to associate a weight to each combinatorial structure of several chords, then take the sum over all subgraphs with this structure in the Gauss diagram. Recently, a family of finite-type invariants of degree two were defined in this way by Chrisman and Dye in \cite{Chr2014}. It is an interesting question to find some explicit virtual knot invariants with higher degrees.

\bibliographystyle{amsplain}

\end{document}